\documentclass[12pt]{amsart}

\usepackage[all,arc]{xy}
\usepackage{enumerate, amsmath, amsthm, amsfonts, amssymb, xy,  mathrsfs, graphicx, paralist, fancyvrb}
\usepackage[usenames, dvipsnames]{xcolor}
\usepackage[margin=1in]{geometry} 
\usepackage[bookmarks, colorlinks=true, linkcolor=blue, citecolor=blue, urlcolor=blue]{hyperref}

\newtheorem{thm}{Theorem}[section]
\newtheorem{cor}[thm]{Corollary}
\newtheorem{prop}[thm]{Proposition}
\newtheorem{lem}[thm]{Lemma}

\theoremstyle{definition}

\theoremstyle{remark}

\makeatletter
\let\c@equation\c@thm
\makeatother

\newcommand{\cE}{\mathcal{E}}

\newcommand{\bF}{\mathbf{F}}

\newcommand{\bK}{\mathbf{K}}
\newcommand{\cK}{\mathcal{K}}

\newcommand{\cL}{\mathcal{L}}

\newcommand{\bN}{\mathbf{N}}

\newcommand{\cO}{\mathcal{O}}

\newcommand{\bP}{\mathbf{P}}

\newcommand{\bQ}{\mathbf{Q}}
\newcommand{\cQ}{\mathcal{Q}}

\newcommand{\bR}{\mathbf{R}}
\newcommand{\cR}{\mathcal{R}}

\newcommand{\rR}{\mathrm{R}}

\newcommand{\cS}{\mathcal{S}}

\newcommand{\cT}{\mathcal{T}}

\newcommand{\cU}{\mathcal{U}}

\newcommand{\bv}{\mathbf{v}}

\newcommand{\dsum}{\bigoplus}
\newcommand{\Gr}{\mathbf{Gr}}

\newcommand{\injects}{\hookrightarrow}
\newcommand{\surjects}{\twoheadrightarrow}

\makeatletter
\newcommand{\RN}[1]{%
  \textup{\uppercase\expandafter{\romannumeral#1}}%
  }
\makeatother

\numberwithin{equation}{section}
\newcommand{\arxiv}[1]{\href{http://arxiv.org/abs/#1}{{\tt arXiv:#1}}}

\title{Equations of Kalman Varieties}

\author{Hang Huang}

\date{May 17, 2017}

\begin{document}

\begin{abstract}
The Kalman variety of a linear subspace is a vector space consisting of all endomorphisms that have an eigenvector in that subspace. We resolve a conjecture of Ottaviani and Sturmfels and give the minimal defining equations of the Kalman variety over a field of characteristic $0$.
\end{abstract}

\maketitle

\section{Introduction}

Let $V$ be a vector space over a field of arbitrary characteristic. For a subspace $L \subsetneqq V$, we can look at all the matrices that have a nonzero eigenvector in $L$. This is called Kalman variety. We define a more general version of it and give some basic properties in Section~\ref{subsec:kaldefn}. Motivated by Kalman's observability condition in control theory \cite{kalman}, Ottaviani and Sturmfels studied their algebraic and geometric properties in \cite{sturmfels}.

In particular, Ottaviani and Sturmfels find minimal generators for the prime ideal of the Kalman variety when dim $L$ = 2. 
The minimal resolution in the case $\dim L = 2$ and minimal generators for the prime ideal in the case $\dim L = 3$ is obtained by Sam in \cite{steven}. Furthermore, he conjectured the existence of a long exact sequence involving the Kalman variety and their higher analogues and proved his conjecture in the case $\dim L = 2,3$. Our main results involve proving Sam's conjecture when we assume the ground field is of characteristic 0 (Theorem~\ref{thm:les}) and use that to get minimal generators for the prime ideal of the Kalman variety in general over a field of characteristic 0 (Corollary~\ref{cor:mineqn}). 

The main tool is the geometric approach to free resolutions via sheaf cohomology (Section~\ref{subsec:KempfVanishing}). This is not a straightforward application since the Kalman variety is not normal if $\dim L > 1$. So the approach only gives information about the normalization of the Kalman variety. But through the existence of the long exact sequence we proved (Theorem~\ref{thm:les}), we are able to extract information about the Kalman variety using information about normalization of all generalized Kalman varieties we got in the geometric approach.

The Kalman variety should be a good testing ground for studying the equations and free resolutions of non-normal varieties. In particular, the approach in Section~\ref{subsec:KempfVanishing} is known to work effectively to study the nilpotent orbits of type A in Lie theory \cite[Chapter 8]{weyman} since they are all normal. But outside type A there are many nilpotent orbits that are not normal and very little is known about the equations and free resolutions of them. Hopefully the insights gained from studying the Kalman varieties will be useful in more complicated situations.

The outline of the article is as follows. In Section~\ref{sec:prelim}, we summarize the properties of Kalman varieties that we will be using, as well as necessary constructions and theorems that we will need. In Section~\ref{sec:repn}, we prove the existence of the long exact sequence and use that to get the minimal defining equations of the Kalman Variety as representations. In Section~\ref{sec:eqn}, we identify those representations with actual minimal defining equations for the Kalman variety.

\subsection*{Acknowledgements.} 

The author would like to send thanks to Steven Sam for bringing this problem into consideration. The author would also like to send special thanks to him for all the fruitful conversations during the writing and conception of this work.

\section{Preliminaries and Notations} \label{sec:prelim}

\subsection{Geometric Approach to Syzygies} \label{subsec:KempfVanishing}

Fix a base field $K$. Let $X$ be a projective variety and $U$ be a vector space. Let $\cS \subset U \times X$ be a subbundle of the trivial bundle whose quotient bundle we denote $\cT$. Let $Z$ be the total space of $\cS$. We have the following map:
\[\xymatrix{
  Z \ar[d]  & \subseteq & U \times X \ar[d]^{p_1} \ar[r]^{p_2} & X \\
  Y := p_1(Z) & \subseteq & U &  \\
} \]

where $p_1$ and $p_2$ are projections. Let $\xi = \cT^*$ and the coordinate ring for $U$ to be $A = \text{Sym}(U^*)$ with grading defined by deg($U^*$) = 1. We denote the ring $A$ with a grading shift to be $A(-i)_d = A_{d-i}$.

\begin{thm} \label{thm:kempf}
We can take a locally free resolution of $\cO_Z$ over $\cO_{U \times X}$ which is given by the Koszul complex $\bK_{\bullet} = \bigwedge^{\bullet} \xi$. Then there exists a sequence of free $A$-modules $\bF_{\bullet}$ which is a minimal free graded representative of $\rR (p_1)_* \bK_{\bullet} = \rR (p_1)_* \cO_Z$ where
\begin{align*}
\bF_i = \dsum_{j \geq 0} H^j (X; \bigwedge^{i+j} \xi) \otimes_{K} A(-i-j).
\end{align*}
Furthermore, if all higher direct image of Sym($\xi^*$) vanish and $p_1 \mid_Z$ is birational, then $\bF_{\bullet}$ is a minimal free resolution of the normalization of $\cO_Y$ which we denoted by $\widetilde{\cO_Y}$.
\end{thm}

\begin{proof}
See \cite[Theorem 5.1.2,5.1.3]{weyman}.
\end{proof}

\subsection{Flag Varieties} \label{subsec:flag}

Given sequence of positive integers $\bv = (v_1,v_2)$, we will use the notation Flag($\bv,L$) to denote the partial flag variety. The typical point of Flag($\bv,L$) is a sequence of subspaces of $L$
\begin{align*}
0 \subset R_{v_1} \subset R_{v_1 + v_2} \subset L.
\end{align*}
The subscript of each subspace denotes its dimension. We will use the symbol $\cR_i$ to denote the tautological subbundle of dimension $i$ on the flag variety. $\cQ_i$ denotes the corresponding quotient bundle. Thus on a flag variety we have the tautological exact sequence $0 \rightarrow \cR_i \rightarrow L \rightarrow \cQ_i \rightarrow 0$. $L$ denotes here the trivial bundle.

\subsection{Kalman Variety} \label{subsec:kaldefn}

We denote $V$ our basic vector space of dimension $n$. $K$ is our base field. We will also fix a subspace $L \subsetneqq V$ of dimension $d$ and pick an integer $1 \leq s \leq d$. The Kalman variety is defined as follows:
\[
\cK_{s,d,n} = \{ \varphi \in \text{End}(V) \mid \exists U \subset L \text{ such that dim}(U) = s \text{ and } \varphi(U) \subset U   \}.
\]
To get set theoretically defining equations, we can pick an ordered basis for $V$ such that the first $d$ vectors form an ordered basis for $L$. In this case, we write the matrix $\varphi$ into a block matrix form: 
$\begin{pmatrix}
\alpha & \beta \\
\gamma & \delta
\end{pmatrix}$. The reduced Kalman matrix is defined by:
\begin{align*}
\begin{pmatrix}
\gamma \\
\gamma \alpha \\
\vdots \\
\gamma \alpha^{d-1}
\end{pmatrix}.
\end{align*}
Then according to \cite[Theorem 4.5]{sturmfels}, the ideal $I_{s,d,n}$ generated by $(d-s+1) \times (d-s+1)$ minors of reduced Kalman matrix defines $\cK_{s,d,n}$ set theoretically. The equations can be easily shown to be not minimal.

Note that using notations in Section~\ref{subsec:KempfVanishing}, we can set 
  \begin{itemize}
  \item $Y = \cK_{s,d,n}$,
  \item $X = \mathbf{Gr}(s,L)$ is the Grassmannian consisting of $s$ dimensional subspaces of $L$,
  \item $U = \text{End}(V)$,
  \item $\cS = \{ (\varphi,U) \mid \varphi (U) \subseteq U  \}$.
  \end{itemize}
In this case, $\xi = \cR_s \otimes (\cQ_s^* \oplus W)$ where $W = (V/L)^*$. According to \cite[Proposition 2.1]{steven}, if char($K$) = 0, higher direct images of $\cS$ vanish and $\bF_{\bullet}$ is a minimal free graded representative of $\tilde{\cO}_{s,d,n}$.

$\cK_{s,d,n}$ is an irreducible variety whose codimension is $s(n-d)$. For $s=d$, $\cK_{s,d,n}$ is defined by $\gamma = 0$ and hence resolved by a Koszul complex. If $s < d$, the non-singular locus of $\cK_{s,d,n}$ = non-normal locus of $\cK_{s,d,n}$ = $\cK_{s+1,d,n}$ \cite[Theorem 4.4]{sturmfels}. In particular if $n > d+1$, $\cK_{s,d,n}$ is not Cohen-Macaulay by Serre's criterion for normality. 
Even though the Kalman varieties are of determinantal type in general, they are not Cohen-Macaulay varieties when $\dim V - 1 > \dim L > 1$. So the resolution is not obtained from the Eagon-Northcott complex.
When $s = 1$ and $n = d+1$, $\cK_{1,d,d+1}$ is a hypersurface and it is Cohen-Macaulay. We know that $I_{1,2,n}$ is prime \cite[Theorem 3.2]{sturmfels} and the minimal free resolution of $\cK_{1,2,n}$ is given in \cite[Theorem 3.3]{steven} for arbitrary characteristic of $K$. According to \cite[Theorem 3.6]{steven}, we can conclude that $I_{1,3,n}$ is prime in arbitrary characteristic. But we do not know the case in general. In this paper, we show that over a field of characteristic $0$, $I_{1,d,n}$ is prime.

\subsection{Schur Functors and Skew Schur Functors} \label{subsec:schur}

If we have a partition $\lambda = (\lambda_1,\ldots,\lambda_n)$, we denote $l(\lambda) = $ the largest $i$ such that $\lambda_i \neq 0$. If $\sum_i \lambda_i = n$, we write $|\lambda|=n$. The dual partition $\lambda^T$ is defined by $\lambda^T_i = \# \{ j \mid \lambda_j \geq i \}$. If the sequence $\lambda = (\lambda_1,\ldots,\lambda_n)$ has repetitions we use exponential notation, for example if $\lambda = (3,3,3,1,1,0,0,0)$ then we write $\lambda=(3^3,1^2,0^3)$.

Let $R$ be a commutative ring and let $U$ be a free $R$-module of finite rank $n$. We define the determinant of $U$ to be det$U$ = $\wedge^n U$. 
We are dealing with characteristic $0$ case. We denote the Schur functor to be $S_{\lambda} U$. Also if $\lambda = (\lambda_1,\ldots,\lambda_n)$ is a dominant integral weight for $GL(n)$, which is the same as a weakly decreasing integer sequence, we have $S_{\lambda} U = S_{(\lambda_1 - \lambda_n, \lambda_2 - \lambda_n,\ldots,0)} U \otimes (\text{det } U)^{\otimes \lambda_n}$. The functor $S_{\lambda}$ is compatible with base change. Hence it makes sense to construct $S_{\lambda} \cU$ when $\cU$ is a locally free sheaf on a scheme.

We will also need skew Schur functors, which have the following interpretation. They are defined for a pair of partitions $(\lambda,\mu)$ where $\mu \subset \lambda$. Equivalently, we want $\mu_i \leq \lambda_i$ for all $i$.  According to \cite[Proposition 2.1.9]{weyman},
\begin{equation*}
S_{\lambda^T / \mu^T} U = \bigwedge^{\lambda_1 - \mu_1} U \otimes \bigwedge^{\lambda_2 - \mu_2} U \otimes \ldots \otimes \bigwedge^{\lambda_s - \mu_s} U / R(\lambda / \mu, U) 
\end{equation*}
where $s = l(\lambda)$ and $ R(\lambda / \mu, U) $ is spanned by the subspaces:
\begin{equation*}
\bigwedge^{\lambda_1 - \mu_1} U \otimes \ldots \otimes \bigwedge^{\lambda_{a-1} - \mu_{a-1}} U \otimes R_{a,a+1}(U) \otimes \bigwedge^{\lambda_{a+2} - \mu_{a+2}} U \otimes \ldots \otimes \bigwedge^{\lambda_s - \mu_s} U
\end{equation*}
for $1 \leq a \leq s - 1$, where $R_{a,a+1}(U)$ is the free $R$ module spanned by the images of the following maps $\theta(\lambda / \mu,a,u,v;U)$ with $u+v < \lambda_{a+1} - \mu_a$:
\[\xymatrix{
  \bigwedge^u U \otimes \bigwedge^{\lambda_a - \mu_a -u + \lambda_{a+1} - \mu_{a+1} - v} U \otimes \bigwedge^v U \ar[d]^{1 \otimes \Delta \otimes 1}  &  \\
  \bigwedge^u U \otimes \bigwedge^{\lambda_a - \mu_a -u} U \otimes \bigwedge^{\lambda_{a+1} - \mu_{a+1} - v} U \otimes \bigwedge^v U \ar[d]^{m_{12} \otimes m_{34}} & \\
  \bigwedge^{\lambda_a - \mu_a} U \otimes \bigwedge^{\lambda_{a+1} - \mu_{a+1}} U   & \\
} \]
Here $\Delta$ is the comultiplication map and $m_{ij}$ is the multiplication of the $i$th tensor component with the $j$th one.

\subsection{Bott's Theorem} \label{subsec:Bott}
For calculation of cohomology of homogeneous vector bundles on the Grassmannian we will use the following version of Bott's Theorem. Given a permutation $\omega$ of size $d$, we define the length of $\omega$ to be $l(\omega) = \# \{i<j \mid \omega (i) > \omega (j) \}$. Furthermore, define $\rho = (d-1,d-2,\ldots,1,0)$. Given any sequence of integers $\alpha$ of length $d$, define $\omega \bullet \alpha = w(\alpha + \rho) - \rho$.

\begin{thm}
Suppose that the characteristic of $K$ is 0. Let $\alpha$,$\beta$ be two partitions and set $\nu = (\alpha,\beta)$. Then exactly one of the following situations occur.
\begin{enumerate}
	\item There exist $\omega \neq$ id such that $\omega \bullet \nu = \nu$. Then all cohomology of $S_{\alpha} \cQ_s \otimes S_{\beta} \cR_s$ vanishes.
	\item There is a unique $\omega$ such that $\eta = \omega \bullet \nu$ is a weakly decreasing sequence. Then $H^{l(\omega)}(\mathbf{Gr}(s,L); S_{\alpha} \cQ_s \otimes S_{\beta} \cR_s) = S_{\eta} L$ and all other cohomology vanish.
\end{enumerate}
\end{thm}

\begin{proof}
See \cite[Corollary 4.1.9]{weyman}.
\end{proof}

\subsection{Cauchy's Formula} \label{subsec:Cauchy}

In our following proof we will use this formula that gives the decomposition of the exterior powers of $E \otimes F$ in terms of Schur functors.
\begin{align*}
\bigwedge^p (E \otimes F) = \sum_{|\lambda| = p} S_{\lambda} E \otimes S_{\lambda^T} F \\
\end{align*}

\begin{proof}
See \cite[Theorem 2.3.2]{weyman}.
\end{proof}

\section{Equations of Kalman Varieties as Representations} \label{sec:repn}

We assume char $K = 0$ throughout the paper. Let $n$ = dim($V$) and $d$ = dim($L$). We will compute minimal defining equations for Kalman varieties as $GL(L) \times GL(W)$ representations using the following theorem.

\begin{thm} \label{thm:les}
Fix $d$, assume char $K$ = 0. For $s = 1, \ldots,d$, let $B_s = \tilde{\cO}_{s,d,n}(-\frac{s(s-1)}{2})$. There is a long exact sequence
\begin{equation*}
0 \longrightarrow \mathcal{O}_{1,d,n} \longrightarrow B_1 \longrightarrow B_2 \longrightarrow \ldots \longrightarrow B_d \longrightarrow 0.
\end{equation*}
\end{thm}

\begin{lem}
We can split the term $\mathbf{F}_i$ of the minimal free resolution $\mathbf{F_{\bullet}}$ of $\tilde{\mathcal{O}}_{s,d,n}$ into three parts:
\begin{align*}
\bF_i = \bF_i^{(\RN{1})} \oplus \bF_i^{(\RN{2})} \oplus \bF_i^{(\RN{3})}
\end{align*}
where
\begin{align*}
\bF_i^{(\RN{1})} =   &\bigoplus_{p = 0}^{s(d-s)} \bigoplus_{\substack{\mu \subseteq s \times (d-s) \\ \mu \subseteq \lambda \subseteq s \times (n-s) \\ |\lambda| = p \\ l(\mu) = s}} H^{p-i}(\mathbf{Gr}(s,L); S_{\lambda}\mathcal{R}_s \otimes S_{\mu^T} \mathcal{Q}_s^*) \otimes S_{\lambda^T / \mu^T }W \otimes A(-p) \\
\bF_i^{(\RN{2})} =  &\bigoplus_{p = 0}^{s(d-s)} \bigoplus_{\substack{\mu \subseteq (s-1) \times (d-s) \\ \mu \subseteq \lambda \subseteq (s-1) \times (n-s) \\ |\lambda| = p }} H^{p-i}(\mathbf{Gr}(s,L); S_{\lambda}\mathcal{R}_s \otimes S_{\mu^T} \mathcal{Q}_s^*) \otimes S_{\lambda^T / \mu^T}W \otimes A(-p) \\
\bF_i^{(\RN{3})} = &\bigoplus_{p = 0}^{s(d-s)} \bigoplus_{\substack{\mu \subseteq (s-1) \times (d-s) \\ \mu \subseteq \lambda \subseteq s \times (n-s) \\ |\lambda| = p \\ l(\lambda) = s }} H^{p-i}(\mathbf{Gr}(s,L); S_{\lambda}\mathcal{R}_s \otimes S_{\mu^T} \mathcal{Q}_s^*) \otimes S_{\lambda^T / \mu^T }W \otimes A(-p).
\end{align*}
When i = 0, we have
\begin{align*}
\bF_0^{(\RN{1})} = & \bigoplus_{p = 0}^{s(d-s)} \bigoplus_{\substack{\mu \subseteq s \times (d-s) \\ |\mu| = p \\ l(\mu) = s}} A(-p) \\
\bF_0^{(\RN{2})} = & \bigoplus_{p = 0}^{s(d-s)} \bigoplus_{\substack{\mu \subseteq (s-1) \times (d-s) \\ |\mu| = p}} A(-p) \\
\bF_0^{(\RN{3})} = & 0.
\end{align*}
\end{lem}

\begin{proof}

Since $\xi = \cR_s \otimes (\cQ_s^{*} \oplus W) $, by Cauchy's Formula in Section~\ref{subsec:Cauchy}, we have
\begin{align*}
\bigwedge^p \xi &= \bigoplus_{|\lambda| = p} S_{\lambda} \cR_s \otimes S_{\lambda^T} (\cQ_s^* \oplus W) \\
			&= \bigoplus_{\substack{|\lambda| = p \\ \mu \subseteq \lambda \subseteq s \times (d-s) \\ \mu \subseteq s \times (n-s) }} S_{\lambda} \cR_s \otimes S_{\mu^T} \cQ_s^* \otimes S_{\lambda^T / \mu^T}W.
\end{align*}

We get the lemma by combining the above equation with the fact that 
$$ \bF_i = \dsum_{p=0}^{s(d-s)} H^{p-i} (\mathbf{Gr} (s,L) ; \bigwedge^p \xi) \otimes_{K} A(-p). $$

In the case when $i=0$, we observe that, according to Borel-Weil-Bott Theorem, the vector bundle $S_{\lambda} \cR_s \otimes S_{\mu^T} \cQ_s^*$, where $ \mu \subseteq \lambda $, has cohomology in degree $|\lambda|$ if and only if $\mu = \lambda$. Furthermore, in this case, we have $ H^{|\lambda|} (\mathbf{Gr}(s,L); S_{\lambda} \cR_s \otimes S_{\mu^{T}} \cQ_s^*) =  H^{|\lambda|} (\mathbf{Gr}(s,L); S_{\mu} \cR_s \otimes S_{\mu^{T}} \cQ_s^*) = K $ for all partitions $\mu = \lambda$. This also follows from \cite[Lemma 2.7]{SSam}.
\end{proof}

\begin{lem} \label{lem:coh}
The following direct summand of $\mathbf{F}_i$ of the minimal free resolution $\mathbf{F_{\bullet}}$ of $\tilde{\mathcal{O}}_{s,d,n}$ is $0$ when $s > i$:
\begin{align*}
\bF_i^{(\RN{3})} = \bigoplus_{p = 0}^{s(d-s)} \bigoplus_{\substack{\mu \subseteq (s-1) \times (d-s) \\ \mu \subseteq \lambda \subseteq s \times (n-s) \\ |\lambda| = p \\ l(\lambda) = s }} H^{p-i}(\mathbf{Gr}(s,L); S_{\lambda}\mathcal{R}_s \otimes S_{\mu^T} \mathcal{Q}_s^*) \otimes S_{\lambda^T / \mu^T }W \otimes A(-p).
\end{align*}
Furthermore, the following direct summand of $\mathbf{F}_s$ of the minimal free resolution $\mathbf{F_{\bullet}}$ of $\tilde{\mathcal{O}}_{s,d,n}$ is:
\begin{align*}
\bF_s^{(\RN{3})} =& \bigoplus_{p = 0}^{s(d-s)} \bigoplus_{\substack{\mu \subseteq (s-1) \times (d-s) \\ \mu \subseteq \lambda \subseteq s \times (n-s) \\ |\lambda| = p \\ l(\lambda) = s }} H^{p-s}(\mathbf{Gr}(s,L)    ; S_{\lambda}\mathcal{R}_s \otimes S_{\mu^T} \mathcal{Q}_s^*) \otimes S_{\lambda^T / \mu^T }W \otimes A(-p) \\
		 =& \dsum_{\substack{\mu \subseteq (s-1) \times (d-s) \\ \lambda = (d-s+1, \mu_1 + 1, \mu_2 + 1, \ldots, \mu_{s-1} + 1 ) }} \bigwedge^d L \otimes S_{\lambda^T / \mu^T } W \otimes A(-|\lambda|).
\end{align*}

\end{lem}

\begin{proof}

We proceed by induction on $s$.

When $s=1$, we have $\mu = (0)$, $\mu^T = (0^{d-1})$, and $\lambda = (p)$. According to Borel-Weil-Bott Theorem in Section~\ref{subsec:Bott}, we can see that the cohomology of $ S_{\lambda} \cR_1 \otimes S_{\mu^T} \cQ_1^*$ exists in degree strictly less than $p$. The cohomology concentrate in degree $p-1$ exactly when $ p = d $. 

Suppose we have such a pair $(\mu,\lambda)$ where $\mu \subseteq s \times (d-s-1)$  and $\mu \subseteq \lambda \subseteq (s+1) \times (n-s-1)$, with the conditions that $l(\lambda) = s+1$ . This is equivalent to $ \lambda_{s+1} \geq 1 $. Suppose further that $l(\mu) \leq s$. This is equivalent to $\mu_1^T \leq s$. Also set $m = \mu_1$, look at the weight vector $\nu = (0^{d-s-1-m},-\mu_m^T,\ldots,-\mu_1^T,\lambda_1,\ldots,\lambda_{s+1}) $ associated to $ S_{\lambda} \cR_{s+1} \otimes S_{\mu^T} \cQ_{s+1}^* $ on $\Gr (s+1,L)$. 

Suppose $\omega$ is the permutation that makes $\omega \bullet \nu$ weakly decreasing. Define $ \nu (1) = (0^{d-s-m-1}, \lambda_1 - \mu_1)$ and $ \nu (2) = (-\mu_m^T + 1, \ldots, -\mu_1^T + 1, \lambda_2, \ldots, \lambda_{s+1}) $ . If $ \omega_i $ is the permutation that makes $ \omega_i \bullet \nu (i) $ weakly decreasing for all $ i = 1,2 $ , then we know $ l(\omega) = l(\omega_1) + l(\omega_2) + \mu_1 $ by first applying the permutation $(d-s-m,d-s-m+1) (d-s-m+1,d-s-m+2) \ldots (d-s-1,d-s)$ to $\nu$. Furthermore, define $ \nu(3) = (0^{d-s-m},-\mu_m^T + 1, \ldots, -\mu_1^T + 1, \lambda_2, \ldots, \lambda_{s+1}) $ and $ \omega_3 $ to be the permutation that makes $ \omega_3 \bullet \nu(3)$ a weakly decreasing sequence. Note we have the weight vector $\nu(3)$ is associated to the bundle $ S_{\lambda'} \cR_s \otimes S_{\mu'^T} \cQ_s^* $ on $\Gr (s,L)$ where $\lambda' = (\lambda_2, \ldots, \lambda_{s+1})$ and $\mu' = (\mu_2, \ldots, \mu_s)$. So $ \mu' \subseteq (s-1) \times (d-s-1) \subseteq (s-1) \times (d-s) $ and $\mu' \subseteq \lambda' \subseteq s \times (n-s-1) \subseteq s \times (n-s) $. Furthermore $l(\lambda') = s$. Hence by induction hypothesis, we have $l(\omega_3 ) \leq |\lambda'| - s = |\lambda| - \lambda_1 - s$. Therefore, $ l(\omega_1) \leq \lambda_1 - \mu_1 -1 $ and $ l(\omega_2) \leq l(\omega_3 ) \leq |\lambda| - \lambda_1 - s $. Combining them, we get $ l(\omega) \leq |\lambda| - (s + 1) $. This completes the induction and implies the cohomology vanishing statement using Borel-Weil-Bott Theorem. 

Moreover, we get that the equality holds exactly when $ \lambda = (d-s, \mu_1 + 1, \mu_2 + 1, \ldots, \mu_s + 1 ) $ in the case $l(\lambda) = s+1$. This gives us the last statement.
\end{proof}

Let $C_s$ be the submodule of $\tilde{\mathcal{O}}_{s,d,n}$ generated by $\bigoplus_{\substack{\mu \subseteq (s-1) \times (d-s)}} A(-|\mu|)$. Also define $C_{d+1} = 0$. Note that $C_1 = \mathcal{O}_{1,d,n}$ and $C_d = \tilde{\mathcal{O}}_{d,d,n}$. We have the following proposition.

\begin{prop} \label{prop:ses}
For $s = 1, \ldots, d$, there are short exact sequences
\begin{equation*}
0 \longrightarrow C_s \longrightarrow \tilde{\mathcal{O}}_{s,d,n} \longrightarrow C_{s+1}(-s) \longrightarrow 0.
\end{equation*}
Furthermore, the term $\mathbf{F}_i^s$ of the minimal free resolution $\mathbf{F}^s_{\bullet}$ of $ C_s $ for $i \leq s-1$ is
\begin{equation*}
\mathbf{F}_i^s =  \bigoplus_{\substack{\mu \subseteq (s-1) \times (d-s) \\ \mu \subseteq \lambda \subseteq (s-1) \times (n-s) }} H^{|\lambda|-i}(\mathbf{Gr}(s,L); S_{\lambda}\mathcal{R}_s \otimes S_{\mu^T} \mathcal{Q}_s^*) \otimes S_{\lambda^T / \mu^T }W \otimes A(-|\lambda|).
\end{equation*}
Moreover, the $s$-th term of $\mathbf{F}^s_{\bullet}$ is given by
\begin{align*}
\mathbf{F}_s^s  = & \bigoplus_{\substack{\mu \subseteq (s-1) \times (d-s) \\ \mu \subseteq \lambda \subseteq (s-1) \times (n-s) }} H^{|\lambda|-s}(\mathbf{Gr}(s,L); S_{\lambda}\mathcal{R}_s \otimes S_{\mu^T} \mathcal{Q}_s^*) \otimes S_{\lambda^T / \mu^T }W \otimes A(-|\lambda|) \\
	   \oplus & \dsum_{k = s}^d \dsum_{\substack{\mu \subseteq (k-1) \times (d-k) \\ \lambda = (d-k+1, \mu_1 + 1, \dots, \mu_{k-1} + 1)}} \bigwedge^d L \otimes S_{\lambda^T / \mu^T} W \otimes A ( -|\lambda| - \frac{(s+k-1)(k-s)}{2} ).
\end{align*}
\end{prop}

Before the actual proof, we will need the following lemma and notation.

\paragraph{Notation} \label{nottilde}
Let $ C = \{ (\mu,\lambda) \mid \mu \subseteq s \times (d-s), \mu \subseteq \lambda \subseteq s \times (n-s), l(\mu) = s \}$. And set $p = |\lambda|$. Assume $\lambda = (\lambda_1,\lambda_2,\ldots,\lambda_s)$ and $\mu = (\mu_1,\mu_2,\ldots,\mu_s)$. We define $\tilde{\lambda} = (\lambda_1 - 1, \lambda_2 - 1, \ldots, \lambda_s - 1, 0)$ and $\tilde{\mu} = (\mu_1 - 1,\ldots,\mu_s - 1, 0)$. We can see that $\tilde{\mu} \subseteq (s+1) \times (d-s-1)$ and $\tilde{\mu} \subseteq \tilde{\lambda} \subseteq (s+1) \times (n-s-1)$. Also $|\tilde{\lambda}| = p - s$ and $\lambda^T / \mu^T = \tilde{\lambda}^T / \tilde{\mu}^T $.

\begin{lem} \label{lem:maps}
There exists a map between $\tilde{\cO}_{s,d,n}$ and $\tilde{\cO}_{s+1,d,n}(-s)$ such that the map identifies 
 $$\displaystyle  H^{j+s}(\mathbf{Gr}(s,L); S_{\lambda}\mathcal{R}_s \otimes S_{\mu^T} \mathcal{Q}_s^*) \otimes S_{\lambda^T / \mu^T }W \otimes A(-|\lambda|)$$ 
 and 
 $$\displaystyle  H^j(\mathbf{Gr}(s+1,L); S_{\tilde{\lambda}} \cR_{s+1} \otimes S_{\tilde{\mu}^T} \cQ_{s+1}^{*}) \otimes S_{\tilde{\lambda}^T / \tilde{\mu}^T} W \otimes A(-|\tilde{\lambda}| - s )$$ 
 for each $j$ and pair $(\lambda, \mu) \in C$ in the minimal free resolution $\bF_{\bullet}$ of $\tilde{\cO}_{s,d,n}$ and $\tilde{\cO}_{s+1,d,n}(-s)$ respectively. And it maps everything else to 0.
 In other words, we are identifying $\bF_i^{(\RN{1})}$ in $\tilde{\cO}_{s,d,n}$ and $\bF_i^{(\RN{2})}(-s)$ in $\tilde{\cO}_{s+1,d,n}(-s)$ for each $i$.
\end{lem}

\begin{proof}[Proof of Proposition~\ref{prop:ses}] 
We will assume the above lemma and show the existence of that short exact sequence first and prove the lemma at the end. The proof of exactness is done by induction on $s$.

When $s = d - 1$, $C_d = \tilde{\mathcal{O}}_{d,d,n}$ is resolved by a Koszul complex. When $i \leq d-1$, the i-th term of its minimal resolution is given by:
\begin{equation*}
\begin{split}
\bF_i	& = \bigwedge^i ( L \otimes W ) \\
	& = \bigoplus_{\substack{|\lambda|=i \\ \lambda \subseteq d \times (n-d) }} S_{\lambda} L \otimes S_{\lambda^T} W \otimes A(-|\lambda|) \\
	& = \bigoplus_{\substack{|\lambda|=i \\ \lambda \subseteq (d-1) \times (n-d) }} S_{\lambda} L \otimes S_{\lambda^T} W \otimes A(-|\lambda|).
\end{split}
\end{equation*}

The $d$-th term of the Koszul complex is given by:
\begin{align*}
\bF_d & = \bigoplus_{\substack{|\lambda|=d \\  \lambda \subseteq d \times (n-d) }} S_{\lambda} L \otimes S_{\lambda^T} W \otimes A(-|\lambda|) \\
      & = (\bigoplus_{\substack{|\lambda|=d \\ \lambda \subseteq (d-1) \times (n-d)}} S_{\lambda} L \otimes S_{\lambda^T} W \otimes A(-|\lambda|)) \oplus (\bigwedge^d L \otimes S^d W \otimes A(-d)).
\end{align*}

The map $\tilde{\mathcal{O}}_{d-1,d,n} \surjects C_{d}(-(d-1))$ will induce a surjection between $\bF_i$ and $\bF_i^d(-(d-1))$ for $i \leq d - 1$. By the long exact sequence on Tor and Lemma~\ref{lem:coh}, we see that
\begin{equation*}
\mathbf{F}_i^{d-1} =  \bigoplus_{\substack{\mu \subseteq (d-2) \times 1 \\ \mu \subseteq \lambda \subseteq (d-2) \times (n-d+1) }} H^{|\lambda|-i}(\mathbf{Gr}(d-1,L); S_{\lambda}\mathcal{R}_{d-1} \otimes S_{\mu^T} \mathcal{Q}_{d-1}^*) \otimes S_{\lambda^T / \mu^T }W \otimes A(-|\lambda|)
\end{equation*}
for $i \leq d - 2$. And when $i = d-1$
\begin{equation*}
\begin{split}
\bF_{d-1}^{d-1} & =  \bigoplus_{\substack{\mu \subseteq (d-2) \times 1 \\ \mu \subseteq \lambda \subseteq (d-2) \times (n-d+1) }} H^{|\lambda|-d+1}(\mathbf{Gr}(d-1,L); S_{\lambda}\mathcal{R}_{d-1} \otimes S_{\mu^T} \mathcal{Q}_{d-1}^*) \otimes S_{\lambda^T / \mu^T }W \otimes A(-|\lambda|) \\
		& \oplus \dsum_{\substack{\mu \subseteq (d-2) \times 1 \\ \mu \subseteq \lambda \subseteq (d-1) \times (n-d+1) \\ l(\lambda) = d-1 }} H^{|\lambda|-d+1}(\Gr (d-1,L); S_{\lambda} \cR_{d-1} \otimes S_{\mu^T} \cQ_{d-1}^*) \otimes S_{\lambda^T / \mu^T} W \otimes A(-|\lambda|) \\
		& \oplus \bigwedge^d L \otimes S^d W \otimes A(-d-(d-1)) \\
	        & =  \bigoplus_{\substack{\mu \subseteq (d-2) \times 1 \\ \mu \subseteq \lambda \subseteq (d-2) \times (n-d+1) }} H^{|\lambda|-d+1}(\mathbf{Gr}(d-1,L); S_{\lambda}\mathcal{R}_{d-1} \otimes S_{\mu^T} \mathcal{Q}_{d-1}^*) \otimes S_{\lambda^T / \mu^T }W \otimes A(-|\lambda|) \\
		& \oplus \dsum_{k = d-1}^d \dsum_{\substack{\mu \subseteq (k-1) \times (d-k) \\ \lambda = (d-k+1,\mu_1+1,\ldots,\mu_{k-1}+1)}} \bigwedge^d L \otimes S_{\lambda^T / \mu^T }W \otimes A(-|\lambda| - \frac{(d+k-2)(k-d+1)}{2}). 
\end{split}
\end{equation*}
Similarly, by induction, we see the proposition is true. 
\end{proof}

Now we will prove the previous lemma and construct the map between $\tilde{\cO}_{s,d,n}$ and $\tilde{\cO}_{s+1,d,n}(-s)$ as follows. 

\begin{proof}[Proof of Lemma~\ref{lem:maps}]
Look at the following commutative diagram:

\[\xymatrix{
 & \mathbf{Flag}(s,s+1,L) \ar[dl]_{\pi_1} \ar[dr]^{\pi_2} & \\
  \mathbf{Gr}(s,L) \ar[dr]_{p_1} & & \mathbf{Gr}(s+1,L) \ar[dl]^{p_2} \\
  & \mathbf{Spec}(K)  & \\
} \]

Let $\cL = \cR_{s+1} / \cR_s$ be a line bundle on $\mathbf{Flag}(s,s+1,L)$. And let $\bP_{\bullet}$ be the Koszul complex on $\mathbf{Gr}(s,L)$ where $\bP_i = \bigwedge^{i} (\cR_s \otimes (\cQ_s^{*} \oplus W))$. Let $\bQ_{\bullet}$ be a similar Koszul complex on $\mathbf{Gr}(s+1,L)$ where $\bQ_i = \bigwedge^i (\cR_{s+1} \otimes ( \cQ_{s+1}^{*} \oplus W ))$.

Let us look at the quotient of $(\pi_1)^* (\bP_{\bullet})$ which we denote $\mathbf{R'}_{\bullet} = \pi_1^*(\mathbf{P'}_{\bullet})$ where
\begin{align*}
\mathbf{P'}_i = \dsum_{\substack{\mu \subseteq \lambda \subseteq s \times (n-s) \\ \mu \subseteq s \times (d-s) \\ l(\mu) = s \\ |\lambda| = i}} S_{\lambda} \cR_s \otimes S_{\mu^T} \cQ_s^* \otimes S_{ \lambda^T / \mu^T} W.
\end{align*}
and
\begin{align*}
\mathbf{R'}_i = \dsum_{\substack{\mu \subseteq \lambda \subseteq s \times (n-s) \\ \mu \subseteq s \times (d-s) \\ l(\mu) = s \\ |\lambda| = i}} S_{\lambda} \cR_s \otimes S_{\mu^T} \cQ_s^* \otimes S_{\lambda^T / \mu^T} W.
\end{align*}

Let $\bR_{\bullet}$ be another complex on $\mathbf{Flag}(s,s+1,L)$ whose $i$th term is defined by
\begin{align*}
\bR_i = \dsum_{\substack{\mu \subseteq \lambda \subseteq s \times (n-s) \\ \mu \subseteq s \times (d-s) \\ l(\mu) = s \\ |\lambda| = i}} S_{\lambda} \cR_s \otimes S_{\tilde{\mu}^T} \cQ_{s+1}^* \otimes \cL^{-s} \otimes S_{\lambda^T / \mu^T} W.
\end{align*}
Here $\tilde{\mu}$ is defined in paragraph \ref{nottilde}.

We can see that this complex is a quotient of $\mathbf{R'}_{\bullet} = \pi_1^*(\mathbf{P'}_{\bullet})$ in the following ways. Let us recall that there is well-known equivalence of categories
\begin{align*}
 [\textrm{homogeneous vector bundles on } GL(L) / P] \xrightarrow{h} [P-\textrm{modules}]
\end{align*}
where $h(V)$ is the fiber at identity and $P$ is the stabilizer of the flag $\langle e_1,\dots,e_s \rangle \subset \langle e_1,\dots,e_{s+1}\rangle$. Under this equivalence we can look at the $P$-module associated to $S_{\mu^T} \cQ_s^*$ on $\mathbf{Flag}(s,s+1,L)$. Let us call it $S_{\mu^T} Q_s^*$. 
Let $e_{s+1}^*, \ldots, e_n^*$ be the weights of the action of $P$ on $Q_s^*$. Then $P$ acting on $S_{\mu^T} Q_s^*$ admits a filtration by weights. Set $Q_i$ be the span of weight vectors in $S_{\mu^T} Q_s^*$ whose weight on the basis $e_{s+1}^*$ is at most $i$. Since $P$ never takes $e_i^*$ to $e_{s+1}^*$ if $i > s+1$, $P$ cannot increase the weight on $e_{s+1}^*$. $Q_i$ is clearly a submodule of $S_{\mu^T}  Q_s^*$. Now take $i = \mu_1^T - 1$ and look at the quotient module $S_{\mu^T} Q_s^* / Q_i$. This is exactly the $P$-module associated to the vector bundle $S_{\tilde{\mu}^T} \cQ_{s+1}^* \otimes \cL^{-s}$ since $\mu_1^T = s$ in our case.

We will show later in Lemma~\ref{lem:maps} that $ \bP_{\bullet} \surjects (\pi_1)_* \bR_{\bullet} $ and $ \rR^s (\pi_2)_* \bR_{\bullet} \injects \bQ_{\bullet}[-s]$. 

Assuming the above, we construct the map first. Now $ \bP_{\bullet} \surjects (\pi_1)_* \bR_{\bullet} $. Since over characteristic $0$ field, the exterior algebra of $\cR_s \otimes (\cQ_s^* \otimes W))$ is semisimple, this surjection splits for each term and makes $(\pi_1)_* \bR_i$ a direct summand of $ \bP_i $ for each $i$. Furthermore, we also have the minimal free resolution of $\tilde{\cO}_{s,d,n}$ is a minimal free graded representative of $\rR (p_1)_* \bP_{\bullet}$. Thanks to the splitting, this will surject onto a minimal free graded representative of $\rR (p_1)_* ((\pi_1)_* \bR_{\bullet}) = \rR (p_1 \circ \pi_1)_* \bR_{\bullet} = \rR (p_2 \circ \pi_2)_* \bR_{\bullet} = \rR (p_2)_* (\rR^s (\pi_2)_* \cR_{\bullet})$. Similarly, the same representative will be a subcomplex of the minimal free graded representative of $\rR (p_2)_* \bQ_{\bullet} [-s]$ which is the minimal free resolution of $\tilde{\cO}_{s+1,d,n}(-s)$. This subcomplex is induced by the map $ \rR^s (\pi_2)_* \bR_{\bullet} \injects \bQ_{\bullet}[-s]$. 

We will compute the terms of the minimal free graded representative of $\rR (p_1 \circ \pi_1)_* \bR_{\bullet}$ using geometric approach to syzygies. According to the relative version of Borel-Weil-Bott Theorem, we know that all higher direct images of $\bR_\bullet$ under $\pi_1$ vanish and $\rR (p_1 \circ \pi_1)_* \bR_{\bullet} = \rR (p_1)_* ((\pi_1)_* \bR_{\bullet}) $. On the one hand,
\begin{equation*}
(\pi_1)_* \bR_i = \dsum_{\substack{\mu \subseteq \lambda \subseteq s \times (n-s) \\ \mu \subseteq s \times (d-s) \\ l(\mu) = s \\ |\lambda| = i}} S_{\lambda} \cR_s \otimes S_{\mu^T} \cQ_s^* \otimes S_{\lambda^T / \mu^T} W.
\end{equation*}
So the $i$th term is
\begin{equation*}
\bF_i = \dsum_{\substack{\mu \subseteq \lambda \subseteq s \times (n-s) \\ \mu \subseteq s \times (d-s) \\ l(\mu) = s }} H^{|\lambda|-i} (\mathbf{Gr}(s,L); S_{\lambda} \cR_s \otimes S_{\mu^T} \cQ_s^*) \otimes S_{\lambda^T / \mu^T} W \otimes A(-|\lambda|).
\end{equation*}
On the other hand,
\begin{equation*}
\rR^s (\pi_2)_* \bR_i = \dsum_{\substack{\tilde{\mu} \subseteq \tilde{\lambda} \subseteq s \times (n-s-1) \\ \tilde{\mu} \subseteq s \times (d-s-1) \\ |\lambda| = i-s }} S_{\tilde{\lambda}} \cR_{s+1} \otimes S_{\tilde{\mu}^T} \cQ_{s+1}^* \otimes S_{\tilde{\lambda}^T / \tilde{\mu}^T} W.
\end{equation*}
So the $i$th term can also be written as
\begin{equation*}
\bF_i = \dsum_{\substack{\tilde{\mu} \subseteq \tilde{\lambda} \subseteq s \times (n-s-1) \\ \tilde{\mu} \subseteq s \times (d-s-1) }} H^{|\tilde{\lambda}| - i} (\mathbf{Gr}(s+1,L); S_{\tilde{\lambda}} \cR_{s+1} \otimes S_{\tilde{\mu}^T} \cQ_{s+1}^*) \otimes S_{\tilde{\lambda}^T / \tilde{\mu}^T} W \otimes A(-|\tilde{\lambda}|-s).
\end{equation*}
We can identify them canonically. Therefore, composing the surjection with the injection gives us the desired map between $\tilde{\cO}_{s,d,n}$ and $\tilde{\cO}_{s+1,d,n}(-s)$.
\end{proof}

Now we will show the remaining part. 

\begin{lem} \label{lem:maps}
We have the following maps: $ \bP_{\bullet} \surjects (\pi_1)_* \bR_{\bullet} $ and $ \rR^s (\pi_2)_* \bR_{\bullet} \injects \bQ_{\bullet}[-s]$.
\end{lem}

\begin{proof}
To show $ \bP_{\bullet} \surjects (\pi_1)_* \bR_{\bullet} $, we apply $(\pi_1)_*$ to the short exact sequence:
\begin{equation*}
0 \longrightarrow \bN_{\bullet} \longrightarrow \pi_1^* (\bP_{\bullet}) \longrightarrow \mathbf{R'}_{\bullet} \longrightarrow 0
\end{equation*}
where $\bN_{\bullet}$ is the kernel whose $i$th term is:
\begin{equation*}
\bN_i = \dsum_{\substack{\mu \subseteq \lambda \subseteq s \times (n-s) \\ \mu \subseteq (s-1) \times (d-s) \\ |\lambda| = i}} S_{\lambda} \cR_s \otimes S_{\mu^T} \cQ_s^* \otimes S_{\lambda^T / \mu^T} W.
\end{equation*}
Note that $\bN_{\bullet} = \pi_1^* (\mathbf{N'}_{\bullet})$ where:
\begin{equation*}
\mathbf{N'}_i = \dsum_{\substack{\mu \subseteq \lambda \subseteq s \times (n-s) \\ \mu \subseteq (s-1) \times (d-s) \\ |\lambda| = i}} S_{\lambda} \cR_s \otimes S_{\mu^T} \cQ_s^* \otimes S_{\lambda^T / \mu^T} W.
\end{equation*}
Since $\pi_1$ is a projective bundle, the structure sheaf of $\mathbf{Flag}(s,s+1,L)$ has no $R^1$ under $\pi_1$. By projection formula, we see that $\rR^1 (\pi_1)_* \bN_i = 0$ for all $i$ and $(\pi_1)_* \pi_1^* \bP_{\bullet} = \bP_{\bullet}$. So according to the long exact sequence associated to $(\pi_1)_*$, we have $\bP_{\bullet} \surjects (\pi_1)_* \bR'_{\bullet}$.

Now look at another short exact sequence:
\begin{equation*}
0 \longrightarrow \bK_{\bullet} \longrightarrow \mathbf{R'}_{\bullet} \longrightarrow \bR_{\bullet} \longrightarrow 0
\end{equation*}
To show $(\pi_1)_* \bK_i = 0$, it is enough to show that $H^0((\mathbf{Flag}(s,s+1,L))_y,\bK_{i,y}) = 0$ for all $y \in \mathbf{Gr}(s,L)$ by Semicontinuity Theorem. Since they are all homogeneous vector bundles, it is enough to check the case when $y$ is identity in $G/P$. Set $\cE = S_{\lambda} \cR_s \otimes S_{\tilde{\mu}^T} \cQ_{s+1}^* \otimes \cL^{-s} \otimes S_{\lambda^T / \mu^T} W$ be a homogeneous vector bundle on $\mathbf{Flag}(s,s+1,L)$. Note that $\pi_1$ realizes $\mathbf{Flag}(s,s+1,L)$ as a $\mathbb{P}^{d-s-1}$ bundle over $\mathbf{Gr}(s,L)$, in this case, when looking at the fiber over identity, since $\cE_y$ is globally generated over $\mathbb{P}^{d-s-1}$ where $y$ is the identity, we have a natural surjection $(\pi_1)_* \cE \otimes \cO_{\mathbb{P}^{d-s-1}} \surjects \cE$. So we have the following short exact sequence
\begin{equation*}
0 \longrightarrow \mathcal{K}_{\lambda, \mu} \longrightarrow (\pi_1)_* \cE \otimes \cO_{\mathbb{P}^{d-s-1}} \longrightarrow \cE \longrightarrow 0. 
\end{equation*}
By construction , $(\pi_1)_*$ applied to the surjection induces an isomorphism and hence $(\pi_1)_* \mathcal{K}_{\lambda, \mu} = 0$. Moreover, $\bK_i = \dsum_{\substack{\mu \subseteq \lambda \subseteq s \times (n-s) \\ \mu \subseteq s \times (d-s) \\ l(\mu) = s \\ |\lambda| = i}} \mathcal{K}_{\lambda, \mu}$ by looking at the corresponding $P$ module. Therefore $(\pi_1)_* \bK_i = 0$ for all $i$. This implies $(\pi_1)_*$ applied to $\mathbf{R'}_{\bullet} \surjects \bR_{\bullet}$ induces an isomorphism, completing the proof of $ \bP_{\bullet} \surjects (\pi_1)_* \bR_{\bullet} $.

To show $ \cR^s (\pi_2)_* \bR_{\bullet} \injects \bQ_{\bullet}[-s]$, we use Serre duality. By Serre Duality, $\cR^s (\pi_2)_* \bR_{\bullet} \cong ((\pi_2)_* (\bR_{\bullet}^* \otimes \omega))^*$ where $\omega$ is the canonical line bundle of the map $\pi_2$. So $\omega = \bigwedge^s \cR_s \otimes \cL^{-s}$.

By direct computation using Borel-Weil, the $i$th term of $((\pi_2)_* (\bR_{\bullet}^* \otimes \omega))^*$ is given by
\begin{equation*}
\dsum_{\substack{\tilde{\mu} \subseteq \tilde{\lambda} \subseteq s \times (n-s-1) \\ \tilde{\mu} \subseteq s \times (d-s-1) \\ |\lambda| = i-s }} S_{\tilde{\lambda}} \cR_{s+1} \otimes S_{\tilde{\mu}^T} \cQ_{s+1}^* \otimes S_{\tilde{\lambda}^T / \tilde{\mu}^T} W.
\end{equation*}
The differential between them is given by the dual of comultiplication coming from exterior algebra, which is the same as the differential for $\bQ_\bullet[-s]$. This is a subcomplex of $\bQ_{\bullet}[-s]$ because the number of rows of a partition can only decrease when you apply the differential and the differentials between these two complexes are compatible.
So this finishes the proof of $ \cR^s (\pi_2)_* \bR_{\bullet} \injects \bQ_{\bullet}[-s]$.
\end{proof}

\begin{cor} \label{cor:eqns}
The minimal defining equations for $\cK_{1,d,n}$ are
\begin{equation*}
\bF_1^1 = \dsum_{s=1}^d \dsum_{\substack{\mu \subseteq (s-1) \times (d-s) \\ \lambda = (d-s+1, \mu_1 + 1, \ldots, \mu_{s-1} + 1)}} \bigwedge^d L \otimes S_{\lambda^T / \mu^T} W \otimes A(-|\lambda|- \frac{s(s-1)}{2}).
\end{equation*}
The projective dimension of $\cK_{1,d,n}$ is $d(n-d) - d + 1$ and its regularity is $\frac{d(d+1)}{2} - 1$.
\end{cor}

\begin{proof}
This follows from Theorem~\ref{thm:les} and Proposition~\ref{prop:ses}.
\end{proof}

\section{Equations of Kalman Variety as $d \times d$ Minors} \label{sec:eqn}

In this section, we are going to show that the minimal defining equations we got from Corollary~\ref{cor:eqns} can be identified as $d \times d$ minors of reduced Kalman matrix. This implies that the ideal $I_{1,d,n}$ is prime when char$(K) = 0$. In order to do this, we need the following lemma.

\begin{lem}
The representations of minimal defining equations of $\cK_{1,d,n}$ given by Corollary~\ref{cor:eqns} can be interpreted as quotients of $\bigwedge^d L \otimes (\bigwedge^{a_0} W \oplus \bigwedge^{a_1} W(-1) \oplus \ldots \oplus \bigwedge^{a_{d-1}} W(-d+1))$ where $a_i$ are nonnegative integers and $\sum_{i} a_i = d$.
\end{lem}

\begin{proof}
Fix $\mu \subseteq (s-1) \times (d-s)$, notice that we have $\sum_i (\lambda_i - \mu_i) = |\lambda| - |\mu| = d$. And by induction on size of $\mu$, we can also get $\sum_i ((i-1)(\lambda_i - \mu_i)) = |\mu| + \frac{s(s-1)}{2}$. Then according to interpretation of skew Schur functors in Section~\ref{subsec:schur}, we have $\bigwedge^d L \otimes S_{\lambda^T / \mu^T} W \otimes A(-|\lambda|- \frac{s(s-1)}{2}) = \bigwedge^d L \otimes S_{\lambda^T / \mu^T} W \otimes A(-|\mu| - \frac{s(s-1)}{2} - d)$ is a quotient of $\bigwedge^d L \otimes (\bigwedge^{\lambda_1 - \mu_1} W \oplus \bigwedge^{\lambda_2 - \mu_2} W(-1) \oplus \ldots \oplus \bigwedge^{\lambda_s - \mu_s} W (-s+1))$. 
\end{proof}

In this case, we can identify all minimal defining equations of $\cK_{1,d,n}$ as $GL(L) \times GL(W)$-representations with some $d \times d$ minors of the reduced Kalman matrix as $GL(L) \times GL(W)$-representations.

To show that those $d \times d$ minors are actually minimal defining equations of $\cK_{1,d,n}$, we need to show further that they do not generate each other. To do this, we need the following lemma.

\begin{lem} \label{lem:relation}
Assume $A, \alpha$ are two $d \times d$ matrices. Then we have $\mathrm{tr}(\bigwedge\limits^i \alpha) \mathrm{det}(A) = \sum\limits_{\substack{I \subseteq \{ 1,\ldots,d \} \\ |I| = i}} \mathrm{det}(A_I)$ where $A_I$ is the matrix obtained by replacing the $i$th row of $A$ by the $i$th row of $A \alpha$ for all $i \in I$.
\end{lem}

\begin{proof}
Let $U$ be a vector space of dimension $d$. The matrix $A$ defines a map $A \colon U \rightarrow U$. And the matrix $\alpha$ also defines a map $\alpha \colon U \rightarrow U$. Now we want to call the map on $\bigwedge^i U$ inducing from $A$ to be: $\bigwedge^i A \colon \bigwedge^i U \rightarrow \bigwedge^i U$. We will be looking at the following two diagrams:
\[\xymatrixcolsep{5pc}\xymatrix{
  \bigwedge^d U \ar[d]^{\Delta} &  & \\
  \bigwedge^{d-i} U \otimes \bigwedge^i U \ar[r]^{\bigwedge^{d-i} A \otimes \bigwedge^i (A \alpha)} & \bigwedge^{d-i} U \otimes \bigwedge^i U \ar[d]^{p} & \\
  & \bigwedge^d U  & \\
} \]
and
\[\xymatrixcolsep{5pc}\xymatrix{
  \bigwedge^d U \ar[d]^{\Delta} & \bigwedge^d U \ar[d]^{\Delta} &  & \\
  \bigwedge^{d-i} U \otimes \bigwedge^i U \ar[r]^{\bigwedge^{d-i} A \otimes \bigwedge^i A } & \bigwedge^{d-i} U \otimes \bigwedge^i U \ar[d]^{p} \ar[r]^{\mathrm{Id} \otimes \bigwedge^i \alpha } & \bigwedge^{d-i} U \otimes \bigwedge^i U \ar[d]^p \\
  & \bigwedge^d U  & \bigwedge^d U \\
} \]
Here the map $\Delta$ is comultiplication and $p$ is multiplication. Now the map $p \circ ( \bigwedge^{d-i} A \otimes \bigwedge^i (A \alpha) ) \circ \Delta$ is a linear map between a 1-dimensional space. So it is defined by multiplication of a number. If we trace the map, the number is exactly $\sum\limits_{\substack{I \subseteq \{ 1,\ldots,d \} \\ |I| = i}} \mathrm{det}(A_I)$.

The calculation is as follows. We will just do the case when $i = 1$ and other cases are similar. Let $e_1,\dots,e_d$ be a basis of $U$, then 
\begin{align*}
p \circ ( \bigwedge^{d-1} A \otimes A \alpha ) \circ \Delta (e_1 \wedge \ldots \wedge e_d) & = p \circ ( \bigwedge^{d-1} A \otimes A \alpha ) (\sum_{j=1}^d (-1)^{d-j} e_1 \wedge \ldots \wedge \hat{e_j} \wedge \ldots \wedge e_d \otimes e_j) \\
										 & = p (\sum_{j=1}^d (-1)^{d-j} A e_1 \wedge \ldots \wedge \hat{A e_j} \wedge \ldots \wedge A e_d \otimes A \alpha e_j ) \\
										 & = \sum_{j=1}^d A e_1 \wedge \ldots \wedge A \alpha e_j \wedge \ldots \wedge A e_d \\
										 & = \sum\limits_{\substack{I = \{ j \} \\ j = 1}}^d \mathrm{det}(A_I) e_1 \wedge \ldots \wedge e_d.
\end{align*}
So the number is exactly $\sum\limits_{\substack{I = \{ j \} \\ j = 1}}^d \mathrm{det}(A_I)$.

On the other hand, the map $p \circ \Delta$ is multiplication by $\binom{d}{i}$. Name its inverse to be $\frac{1}{\binom{d}{i}} \colon \bigwedge^d U \rightarrow \bigwedge^d U$. This is multiplication by $\frac{1}{\binom{d}{i}}$. If we trace the map $p \circ (\bigwedge^{d-i} A \otimes \bigwedge^i A) \circ \Delta$, it is defined by multiplication of the number $\binom{d}{i} \cdot \mathrm{det}(A)$. If we trace the map $p \circ ( \mathrm{Id} \otimes \bigwedge^i \alpha ) \circ \Delta$, it is defined by multiplication of the number $\mathrm{tr}(\bigwedge\limits^i \alpha)$. Therefore we have
\begin{align*}
p \circ ( \bigwedge^{d-i} A \otimes \bigwedge^i (A \alpha) ) \circ \Delta & = p \circ ( \mathrm{Id} \otimes \bigwedge^i \alpha ) \circ ( \bigwedge^{d-i} A \otimes \bigwedge^i A ) \circ \Delta \\
									  & = (p \circ (\mathrm{Id} \otimes \bigwedge^i \alpha ) \circ \Delta) \circ (\frac{1}{\binom{d}{i}}) \circ (p \circ ( \bigwedge^{d-i} A \otimes \bigwedge^i A ) \circ \Delta).
\end{align*}
By composing those three maps as in the last part of the equation, we get $p \circ ( \bigwedge^{d-i} A \otimes \bigwedge^i (A \alpha) ) \circ \Delta$ is also defined by multiplication of $\mathrm{tr}(\bigwedge\limits^i \alpha) \mathrm{det}(A)$, completing the proof.
\end{proof}

Now we finish the section by the following corollary.

\begin{cor} \label{cor:mineqn}
If char($K$) = 0, $I_{1,d,n}$ is prime. Furthermore, if we identify those $d \times d$ minors in reduced Kalman matrix, which as $\mathrm{GL}(L) \times \mathrm{GL}(W)$ representations, are the same as 
\begin{equation*}
\dsum_{s=1}^d \dsum_{\substack{\mu \subseteq (s-1) \times (d-s) \\ \lambda = (d-s+1, \mu_1 + 1, \ldots, \mu_{s-1} + 1)}} \bigwedge^d L \otimes S_{\lambda^T / \mu^T} W \otimes A(-|\lambda|- \frac{s(s-1)}{2})
\end{equation*}
then they define a minimal set of defining equations of $\cK_{1,d,n}$.
\end{cor}

\begin{proof}
Now we split the proof into two parts. In the first part, we are going to show those $d \times d$ minors in reduced Kalman matrix as stated in the above Corollary form a set of minimal equations in $I_{1,d,n}$. In the second part, we will show they define a minimal set of defining equations of $\cK_{1,d,n}$.

In order to show those equations are a set of minimal equations, we will show that they do not generate each other in $I_{1,d,n}$. Note that we are only using $\alpha$ and $\gamma$ in the reduced Kalman matrix. So we are done if we can show that the $GL(L) \times GL(W)$ representations we got in Corollary~{\ref{cor:eqns}} will not land in each other when we tensor it with $\mathrm{Sym}^{\bullet}(L \otimes L^* \oplus L \otimes W) = \mathrm{Sym}^{\bullet} (L \otimes L^*) \otimes \mathrm{Sym}^{\bullet} (L \otimes W)$. Also notice that the representations we got in Corollary~{\ref{cor:eqns}} are $\bigwedge^d L$ as $GL(L)$ representations. Therefore we only have to deal with tensoring it with $\mathrm{Sym}^{\bullet} (L \otimes L^*)$. Moreover, since $L$ is $d$-dimensional, the only $\bigwedge^d L$ we can get in $\bigwedge^d L \otimes \mathrm{Sym}^{\bullet} (L \otimes L^*)$ are when we tensor $\bigwedge^d L$ with the ring of invariants $\mathrm{Sym}^{\bullet} (L \otimes L^*)^{GL(L)}$. It is a well-known fact that the ring of invariants is a polynomial ring in $\mathrm{tr}(\bigwedge^i \alpha)$ for $i=1, \ldots, d$.

Now if we take any minors which is in $\bigwedge^d L \otimes (\bigwedge^{a_0} W \oplus \bigwedge^{a_1} W(-1) \oplus \ldots \oplus \bigwedge^{a_{d-1}} W(-d+1))$, we can take $A$ to be the submatrix we are taking determinant, by the Lemma~{\ref{lem:relation}}, we see that the only linear combination of $d \times d$ minors of the reduced Kalman matrix that it could generate when multiplying with $\mathrm{tr}(\bigwedge^i \alpha)$ for $i=1, \ldots, d$ are in the image of the map 
\[
\xymatrix{
  \bigwedge^d L \otimes (\bigwedge^{a_0} W \oplus \bigwedge^{a_1} W(-1) \oplus \ldots \oplus \bigwedge^{a_{d-1}} W(-d+1) ) \ar[d]^{\Delta}  \\
  \sum\limits_{\substack{\sum\limits_{j = 0}^{d-1} b_j = i \\ 0 \leq b_j \leq a_j}} (\bigwedge^d L \otimes (\bigwedge^{a_0 - b_0} W \oplus \bigwedge^{a_1 - b_1 + b_0} W(-1) \oplus \ldots \\
  \oplus \bigwedge^{a_{d-1} - b_{d-1} + a_{d-2}} W(-d+1) \oplus \bigwedge^{b_{d-1}} W(-d) )) 
} 
\]
So they all either land in the kernel of the quotients which are the skew Schur functors we got in Corollary~{\ref{cor:eqns}} or outside the skew Schur functors. Hence those equations which we identified with those skew Schur functors cannot generate each other and they give a set of minimal equations inside $I_{1,d,n}$. 

Now they form a set of minimal equations in $I_{1,d,n}$. If those equations are not minimal defining equations of $\cK_{1,d,n}$, since they are contained in the defining ideal, in order to generate them, we need to either throw in equations of lower degree or increase the number of equations. This contradicts the result in Corollary~{\ref{cor:eqns}} giving the number of minimal generating equations in each degree.

Therefore, we have if char($K$) = 0, $I_{1,d,n}$ is reduced. And since each of them defines an irreducible variety set theoretically, they are prime.
\end{proof}

\small \noindent Hang Huang, Department of Mathematics,
University of Wisconsin Madison, Madison, WI 53706 \\
{\tt hhuang235@math.wisc.edu}, \url{http://math.wisc.edu/~hhuang235/}

\end{document}